\documentclass{amsart}
         \pdfoutput=1
\usepackage{gensymb}%
\usepackage{amssymb}
\usepackage{amsfonts}%
\usepackage{amsmath}
\usepackage{url}%
\usepackage{hyperref}
\newtheorem{theorem}{Theorem}[section]
\newtheorem{corollary}[theorem]{Corollary}
\newtheorem{proposition}[theorem]{Proposition}
\newtheorem{lemma}[theorem]{Lemma}

\theoremstyle{definition}
\newtheorem{example}[theorem]{Example}

\theoremstyle{remark}

\newcommand{\coker}{\operatorname{coker}}
\newcommand{\im}{\operatorname{im}}
\newcommand{\Spec}{\operatorname{Spec}}
\newcommand{\Proj}{\operatorname{Proj}}
\newcommand{\Hom}{\operatorname{Hom}}
\newcommand{\Pic}{\operatorname{Pic}}

\newcommand{\B}{\operatorname{B}}
\newcommand{\HH}{\operatorname{H}}

\def \ScrO {\mathcal{O}} 
\title{Division Algebras and Quadratic Reciprocity}

\author{Timothy J. Ford}%
\address{Department of Mathematics,
  Florida Atlantic University,
  Boca Raton, Florida 33431}%
\email{\tt Ford@fau.edu}%
\date{June 8, 1995. Updated September 9, 1997. Last updated \today}%
\thanks{Preliminary Report}%
\begin{document}
\begin{abstract}
  The Grothendieck and Artin-Mumford exact sequences for the Brauer
  group of a function field in 1 or 2 variables are applied to derive
  reciprocity laws for $q$th power residues.
\end{abstract}

\subjclass[2010]{16K50; %
  Secondary 14F22, 12E15, 14F20, 11R52}%
\keywords{Brauer group, division algebra}%

\maketitle %

\section{Introduction}
For positive integers $n$ and $m$ define the Legendre symbol
\[
(n/m) = \left\{
\begin{array}{ll}
1 & \text{if } n \text{ is congruent to a square modulo } m \\
-1 & \text{otherwise}
\end{array}
\right. \text{ .}
\]
If $p$ and $q$ are distinct odd prime numbers, then the Quadratic
Reciprocity Formula (conjectured by Euler, proved by Gauss
\cite{MR0197380}) is 
\begin{equation}
\label{ab1}
(p/q)(q/p) = (-1)^{\frac{p-1}{2}\frac{q-1}{2}} \text{ .}
\end{equation}
This formula tells one how to determine the value of the symbol $(p/q)$, if
that of $(q/p)$ is known.

For a global field $K$ (that is, either an algebraic number field or an
algebraic function field in one variable over a finite constant field) the
structure of the Brauer group $B(K)$ was completely determined in the 1930s
by the work of Albert, Brauer, Hasse and E. Noether \cite[Chapter 7]{D:A}.
As a consequence of
their exact sequence describing $B(K)$ when $K$ is the field of rational
numbers $\mathbb Q$, it is possible to derive \eqref{ab1}.

In the 1960s and 1970s, turning the cohomological crank on the engine of
Algebraic Geometry, Grothendieck, M. Artin and Mumford derived exact
sequences for 
the Brauer group of function fields for varieties of dimension 1 and 2.
These Brauer group theorems can be viewed as generalizations of the results
of class field theory and furthermore can  be thought of as providing laws
of $q$th degree reciprocity.

Let us say what we mean by a $q$th degree reciprocity formula.
Suppose ${R}$ is a noetherian integral domain.
Let $f$ and $g$ be nonzero elements of $R$. Define a Legendre symbol
\begin{equation}
\label{eq4}
(f/g)_q = \left\{
\begin{array}{ll}
1 & \text{if } f \text{ is congruent to a $q$th power modulo } g \\
-1 & \text{otherwise}
\end{array}
\right. \text{ .}
\end{equation}
A $q$th degree reciprocity formula should be a formula allowing one to
compute $(f/g)_q$ in terms of $(g/f)_q$.

We  show that the Grothendieck and Artin-Mumford sequences
can sometimes be employed to achieve $q$th degree reciprocity for (a)
polynomials  in one variable over a field,
(b)  power series in two variables over an
algebraically closed field  and (c) 
power series in the variable $y$ with coefficients that
are polynomials in $x$ over an algebraically closed field.

Let $X$ be a regular, integral, locally noetherian, quasi-compact scheme
with generic stalk $K$.
Let $X_1$ denote the set of points of $X$ of codimension 1. Usually we
assume $X$ to be $\Spec{R}$ for a noetherian regular integral domain $R$
with quotient field $K$. In this case, $X_1$ consists of those prime ideals
in $R$ of height 1.

Throughout cohomology groups and sheafs will be for the \'etale topology.
The sheaf of units on $X$ is denoted $\mathbb G_m$. The sheaf $\mu_n$  of
$n$th roots of unity is the kernel of the
$n$th power map
\[
1 \to \mu_n \to \mathbb G_m \xrightarrow{n} \mathbb G_m \text{ .}
\]
Let $\mu = \bigcup_n \mu_n$ and $\mu(-1) = \bigcup_n \Hom(\mu_n,\mathbb Q/\mathbb
Z)$. 
If $X$ is a scheme over $\mathbb Z[1/n][\zeta]$ for a primitive $n$th root of
unity $\zeta$, then $\mu_n$ is isomorphic to the constant sheaf $\mathbb Z/n$
(noncanonically). The group $\HH^1(X,\mathbb Z/n)$ parametrizes the cyclic
Galois extensions of $X$ with group $\mathbb Z/n$.
The cohomology groups for $\mathbb G_m$ in the lowest degrees have the
following descriptions. The global sections of $\mathbb G_m$ make up the group
$\HH^0(X,\mathbb G_m)$.
The global units are those units in $K$ that are defined at each point of $X$.
The group $\HH^1(X,\mathbb G_m) = \Pic{X}$ is the Picard group of invertible
$\ScrO_X$-modules. The group  $\HH^2(X,\mathbb G_m)$ is the cohomological Brauer
group. If $X$ is an affine scheme (for example) it is known by the
Gabber-Hoobler Theorem \cite{MR611868} that the Brauer group $B(X)$ of classes
of Azumaya $\ScrO_X$-algebras is isomorphic under a canonical embedding to
the torsion subgroup of $\HH^2(X,\mathbb G_m)$.

Given $\alpha$ and $\beta$ in $K^\ast$ let $n$ be a positive integer that is
invertible in $K$ and let $\zeta$ be a primitive $n$th root of unity in
$K$. The symbol algebra $(\alpha,\beta)_n$
is the associative $K$-algebra
generated by elements $u$, $v$ subject to the relations $u^n= \alpha$, $v^n
= \beta$ and $uv = \zeta vu$. The symbol algebra $(\alpha,\beta)_n$  is
central simple over $K$ and represents a class in ${_n\B(K)}$.
This agrees with the cyclic crossed product algebra
$(K(\alpha^{1/n})/K,\sigma,\beta)$ for the cyclic Galois extension of
degree $n$
$K[u]/(u^n-\alpha)$ whose group is generated by $\sigma$ and with factor
set $\beta$ \cite[Section~30]{R:MO}.

The following theorem gives
the fundamental connection between $q$th power residues and division
algebras. 
\begin{theorem}\label{th3}
Let $\alpha$ and $\beta$ be elements of $R$ where $R$ is a noetherian,
regular, integral domain. Let $X=\Spec{R}$.
If 2 is invertible in $R$ and 
$\alpha$ is a square modulo $\beta$, then the ramification divisor of
the symbol algebra $(\alpha,\beta)_2$ is a subset of the divisor of
$\alpha$. 
For any prime number
$q$ that is invertible in $R$, if $R$ contains a primitive $q$th root of
unity and
$\alpha$ is a $q$th power modulo $\beta$, then the ramification divisor of
the symbol algebra $(\alpha,\beta)_q$ is a subset of the divisor of
$\alpha$. 
\end{theorem}

Before proving Theorem~\ref{th3}, we review the theory underlying the
definition of the ramification divisor of a division algebra.

Given a  finite dimensional central $K$-division algebra $D$, it is
possible to measure the ramification of $D$ at any point $x \in X_1$.
The local ring $\ScrO_{X,x}$ at $x$ is a discrete valuation ring.
Let $\nu$ be the discrete
rank-1 valuation on $K$ corresponding to the local ring $\ScrO_{X,x}$.
Let $k(x)$ denote the residue field at $x$.
Assume that $k(x)$ is perfect. (If $k(x)$ is not perfect, the following
still works if $(D:K)$ is prime to the characteristic of $k(x)$.)
The theory of maximal orders \cite[Section~5.7]{S:TV} associates to $D$ a
cyclic extension $L$ of $k(x)$.
Let $K^{\nu}$ be the completion of $K$ and $D^{\nu}$ the division algebra
component of $D \otimes K^{\nu}$. Let $A$ be a maximal order for $D^{\nu}$
in the complete local ring $\ScrO^{\nu}_{X,x}$ and let $A(x) = A \otimes
k(x)$ be the algebra of residue classes. Then $A(x)$ is a central simple
algebra over $L$ for some cyclic Galois extension $L/k(x)$. The cyclic
extension $L/k(x)$ represents a class in $\HH^1(k(x),\mathbb Q/\mathbb Z)$. 
The Brauer group of the field $K^{\nu}$ factors into
\[
\B(K^ \nu) \cong \B(k(x)) \oplus \HH^1(k(x),\mathbb Q/ \mathbb Z) \text{ .}
\]
Every division algebra $D^ \nu$ has a factorization $D^ \nu = D_u \otimes
(L^ \nu / K^ \nu, \sigma, \pi)$. The division algebra $D_u$ has the
property that the corresponding algebra of residue classes $A_u \otimes
k(x)$ is a $k(x)$-central division algebra hence represents a class in
$B(k(x))$. Every cyclic Galois extension of $k(x)$ is the algebra of
residue classes for a cyclic Galois extension $L^ \nu / K^ \nu$ (with group
$\langle \sigma \rangle$). The factorization of $D^ \nu$ is unique up to
choice of local parameter $\pi$.

The assignment $D \mapsto L$ induces a group homomorphism
\begin{equation}
\label{eq1}
\B(K) \to \HH^1(k(x),\mathbb Q/\mathbb Z)
\end{equation}
for each discrete rank-1 valuation $\nu$ on $K$ corresponding to a point $x
\in X_1$. We call $L$ the ramification of $D$ along $x$.
The algebra $D$ will ramify at only finitely many $x \in X_1$. Those $x$
for which the cyclic extension $L/k(x)$ is nontrivial make up the so-called
ramification divisor of $D$. So \eqref{eq1} induces a homomorphism
\begin{equation}
\label{eq2}
\B(K) \xrightarrow{a} \bigoplus_{x \in X_1} \HH^1(k(x),\mathbb Q/\mathbb Z) \text{ .}
\end{equation}

Let $n$ be a positive integer. If $K$ and $k(x)$ both contain $1/n$ and a
primitive $n$th root of unity $\zeta$, this homomorphism agrees with the
tame symbol. On the symbol algebra $(\alpha,\beta)_n$ over $K$, the value
of the homomorphism  \eqref{eq1} is the cyclic extension $L/k(x)$ which is
obtained by adjoining the $n$th root of
\begin{equation}
\label{eq3}
(-1)^{\nu(\alpha)\nu(\beta)}\alpha^{\nu(\beta)}/\beta^{\nu(\alpha)}
\end{equation}
to $k(x)$.
The divisor of $\alpha$ is the set of $x\in X_1$ where $\alpha$ has nonzero
valuation. 

\begin{proof}[Proof of Theorem~\ref{th3}]
By \eqref{eq3}, the ramification divisor of the
symbol algebra $(\alpha,\beta)_q$ is a subset of the set of all
prime divisors 
$x\in X_1$ where $\alpha$ or $\beta$ has nonzero valuation. 
Assume $\alpha$ is a $q$th power modulo $\beta$. Suppose at the prime divisor
$x\in X_1$,  $\nu(\alpha)=0$ and $\nu(\beta)>0$. Then $\alpha$ is a $q$th
power in the residue field $k(x)$ hence \eqref{eq3} defines a trivial
cyclic extension. 
\end{proof}

Let $K$ be an algebraic number field. If the direct sum in
\eqref{eq2} is taken over all primes $x$ of $K$, both finite and infinite,
then the Hasse Principle says the map $a$ in \eqref{eq2} is injective. For
general $K$, this will not be the case.

The Quadratic Reciprocity Formula \eqref{ab1} arises by looking at the
cokernel of the ramification map $a$. For general $K$ the description of
$\coker{a}$ will not lead to a $q$th degree reciprocity formula that is of
much practical importance. But in many instances there are
descriptions of $\coker{a}$ that have useful interpretations.

\section{The One-dimensional Case.}
Let us review how the quadratic reciprocity formula for integers follows
{f}rom a description of $\coker{a}$ in \eqref{eq2}.
If $K$ is a global field, then the residue field $k(x)$ is a
finite field (if the prime $x$ is finite) or $\mathbb R$ or $\mathbb C$ (if $x$
is an infinite prime). Now $\HH^1(\mathbb R, \mathbb Q/\mathbb Z) = \mathbb Z/2$ and
$\HH^1(\mathbb C, \mathbb Q/\mathbb Z) = 0$. If $k(x)$ is a finite field, then the
Galois group of $k(x)$ is isomorphic to $\hat{\mathbb Z}$, the profinite
completion of $\mathbb Z$, hence $\HH^1(k(x), \mathbb
Q/\mathbb Z) = \mathbb Q/ \mathbb Z$. The group $\HH^1(\mathbb R, \mathbb Q/\mathbb Z)$ is
identified with the subgroup of order 2 in $\mathbb Q/\mathbb Z$. So for any
prime $x$ of $K$ there is a homomorphism $\HH^1(k(x), \mathbb Q/\mathbb Z) \to
\mathbb Q/\mathbb Z$. From Class Field Theory, the sequence
\begin{equation}
\label{eq5}
0 \to \B(K) \xrightarrow{a} \bigoplus_{x \in X_1} \HH^1(k(x),\mathbb Q/\mathbb Z)
\xrightarrow{r}  \mathbb
Q/\mathbb Z \to 0
\end{equation}
is exact, where $K$ is any global field and $X_1$ denotes the set of all
primes of $K$.

The Quadratic Reciprocity Formula \eqref{ab1} follows from a computation of
$r \circ a$ applied to a symbol algebra $(p,q)_2$ when $K = \mathbb Q$. Let
$p$ and $q$ be distinct odd prime numbers. Consider the algebra $(p,q)_2$
over $\mathbb Q$. The tame symbol \eqref{eq3} will be trivial at every odd
prime different from $p$ and $q$ because $\nu(p)=\nu(q)=0$. At the prime
$p$, the symbol \eqref{eq3} is $1/q$ and the residue field is $\mathbb Z/p$.
So the ramification is trivial if and only if $q$ is a square modulo $p$.
That is, the ramification (written multiplicatively) is $(q/p)$. Likewise,
at the prime $q$, the ramification is $(p/q)$.
At the infinite prime the complete local ring is the field $\mathbb R$. Since
$p$ and $q$ are both positive integers the algebra $(p,q)_2$ is split over
$\mathbb R$ hence is unramified at this prime.
The only other prime where
ramification can occur is the prime 2. This is the wild case, and the tame
symbol does not apply. It turns out that the ramification over $\mathbb Z/2$
is given by the formula
$(-1)^{\frac{p-1}{2}\frac{q-1}{2}}$. Because the composite map $r \circ a$
applied to a symbol algebra $(p,q)_2$ is 0, the formula \eqref{ab1} holds.
For more details, the reader is referred to \cite[XIV, section~4]{MR554237}.

The proof of the next result of Grothendieck can be found
in \cite[Proposition~2.1]{G:GBIII} or \cite[p. 107, Example~2.22,
case(a)]{M:EC}.

\begin{theorem}
\label{th1}
Let $X$ be a regular integral scheme of dimension 1. Let $K = K(X)$ be the
stalk at the generic point of $X$ and $X_1$ the set of closed points of
$X$. Suppose that for each $x \in X_1$, the residue field $k(x)$ is
perfect. Then there is an exact sequence
\begin{multline}
\label{eq7}
0 \to \HH^2(X,\mathbb G_m) \to \HH^2(K,\mathbb G_{m,K}) \xrightarrow{a} \bigoplus_{x \in X_1}
\HH^1(k(x),\mathbb Q/ \mathbb Z) \xrightarrow{r} \HH^3(X,\mathbb G_m) \\ \to \HH^3(K,\mathbb G_{m,K})
\text{ .}
\end{multline}
If we do not assume the residue fields are perfect, the sequence is still
exact for the $q$-primary components of the groups, for any prime $q$
distinct from the residue characteristics of $X$.
\end{theorem}

The first 2 groups in \eqref{eq7} are the Brauer groups of $X$ and $K$
respectively. The map $a$ in \eqref{eq7} is  ``the ramification map''
\eqref{eq2}.
The fact that in \eqref{eq7} $r \circ a $ is the zero map can be thought of
as a quadratic reciprocity law for elements of order 2, or a $q$th degree
reciprocity law for elements of order $q$.
But to have practical implications, one
must know that  $\HH^1(k(x),\mathbb Q/ \mathbb Z) \xrightarrow{r} \HH^3(X,\mathbb G_m)$ is
injective for some $x \in X_1$. We will prove a lemma showing this is the
case when $X$ is the projective line over a field $k$ and  $x$ is a point
with residue field $k$.

Let $k$ be a field with characteristic $p$ ($p=0$ is allowed)
and $X = \mathbb P^1_k =
\Proj{k[x_0,x_1]}$. Let $x$ be a closed point of $X$ with residue field
$k(x) = k$. There is an open cover of $X$ by the affine sets
$\Spec{k[x_0/x_1]}$,  $\Spec{k[x_1/x_0]}$. The Mayer-Vietoris sequence (for
the \'etale topology and the sheaf of units $\mathbb G_m$) is \cite[p.
110]{M:EC}
\begin{multline}
\label{eq6}
1 \to \HH^0(X,\mathbb G_m) \to  k[x_0/x_1]^\ast \times  k[x_1/x_0]^\ast \to
k[x_0/x_1,x_1/x_0]^\ast  \\
\to \Pic{X} \to \Pic{k[x_0/x_1]} \oplus \Pic{k[x_1/x_0]} \to
\Pic{k[x_0/x_1,x_1/x_0]} \\
\to \B(X) \to \B(k[x_0/x_1]) \oplus \B(k[x_1/x_0]) \to
\B(k[x_0/x_1,x_1/x_0]) \\
\to \HH^3(X,\mathbb G_m) \to \HH^3(k[x_0/x_1],\mathbb G_m) \oplus
\HH^3(k[x_1/x_0],\mathbb G_m)
\to \dots  \text{ .}
\end{multline}
(Actually we only need the 4 Brauer group terms and the first $\HH^3$ term but
include the rest for curiosity's sake.) We write $(\cdot)^\ast$ for the group
of units in a ring.
Since $k$ is a field, $k[T]^\ast = k^\ast$ so
$\HH^0(X,\mathbb G_m) = k^\ast$
and $k[T,T^{-1}]^\ast  = k^\ast  \times \langle T \rangle \cong k^\ast  \times \mathbb
Z$. Since $k[T]$ is factorial, $\Pic{k[T]} = \Pic{k[T,T^{-1}]} = 0$.
Therefore $\Pic{X} \cong \mathbb Z$ and is generated by the divisor class of
the closed point $x$ associated to some linear form $ax_0+bx_1$. From
\cite[p. 164, 6)]{MR770588}
if $n$ is a positive integer not divisible by $p$, then
${_n\B(k[T])} = {_n\B(k)}$ hence ${_n\B(X)} = {_n\B(k)}$.
By \cite[Theorem~2.4]{MR770588} ${_n\B(k[T,T^{-1}])} \cong {_n\B(k)} \oplus
{_n\HH^1(k,\mathbb Q/ \mathbb Z)}$.
We write ${_n(\cdot)}$ for the subgroup annihilated by $n$.
The group ${_n\HH^1(k(x),\mathbb Q/ \mathbb Z)}$
measures the Galois extensions of the field $k(x)$ with group $\mathbb Z/n$
\cite[III, section 4]{M:EC}.
A section to the  epimorphism
${_n\B(k[T,T^{-1}])} \to {_n\HH^1(k,\mathbb Q/ \mathbb Z)}$ is defined by
mapping a cyclic Galois extension $L/k$ with Galois group $\langle
\sigma \rangle$ to the Brauer class of the
cyclic crossed product algebra \cite[Section~30]{R:MO}
$(L(T)/k(T),\sigma,T)$. This algebra is unramified on $\Spec{k[T,T^{-1}]}$
so by \eqref{eq7} represents a class in $\B(k[T,T^{-1}])$.

The reader may wish to compare our computation of $\B(k(y))$ to that given
in \cite[Theorem, p. 51]{FS:Bgrff} and \cite[Proposition~4.1]{AB:Bgd}.
Because we want a $q$th degree reciprocity law, it is necessary to include the
point at infinity, and know that the Gysin map of Lemma~\ref{lem1} is
an injection at that point.

\begin{lemma}
\label{lem1}
Let $k$ be a field and $n$ a positive integer
invertible in $k$. Let $x$ be a closed point of $X = \mathbb P^1_k$ with
residue field $k(x) = k$.
There exists a natural Gysin map
\[
{_n\HH^1(k(x),\mathbb Q/ \mathbb Z)} \xrightarrow{r} {_n\HH^3(X,\mathbb G_m)}
\]
which is injective.
\end{lemma}
\begin{proof}
The injectivity follows from the preceding discussion. To see that the map
$r$ is the Gysin map follows from comparing the above computation to that
in \cite[Corollary~2]{F:BgLp}.
\end{proof}

\begin{example}
\label{ex7}
We interpret the above for the projective line over $k=\mathbb R$. Up to
associates, the irreducible polynomials in $\mathbb R[y]$ are of the form
$y-a$ or $(y-a)^2+b^2$ for $a\in\mathbb R$ and $b\in\mathbb R^\ast$. The residue
fields at the prime divisors of $\mathbb P^1_{\mathbb R}$ are $\mathbb R$ and $\mathbb
C$ depending on whether the maximal ideal is generated by a linear or
quadratic polynomial. Therefore 
\[ \HH^1(k(x),\mathbb Q /\mathbb Z) \cong \left\{ 
\begin{array}{ll}
\mathbb Z/2 & \text{if } k(x) = \mathbb R \\
0 & \text{if } k(x) = \mathbb C
\end{array}
\right. \text{ .}
\]
Therefore an algebra class in $\B(\mathbb R(y))$ ramifies only at points of
$\mathbb P^1_{\mathbb R}$  with residue field $\mathbb R$. Combining
Theorem~\ref{th1} and Lemma~\ref{lem1}, we have the exact sequence
\begin{equation}
\label{eq11}
0 \to \B(\mathbb R) \to \B(\mathbb R(y)) \xrightarrow{a} \bigoplus_{x \in\mathbb R\cup\{\infty\}}
 \mathbb Z/2 \xrightarrow{r} \mathbb Z/2 \to 0
\end{equation}
where $r$ is the summation map. If $a_1$ is not equal to $a_2$, then by
\eqref{eq3}, the ramification of the symbol algebra $(y-a_1,y-a_2)_2$ at
the prime $y-a_2$ is equal to $a_2-a_1$. So $(y-a_1,y-a_2)_2$ ramifies at
$y-a_2$ if and only if $a_2-a_1<0$. So $(y-a_1,y-a_2)_2$ ramifies at
exactly one of the 2 primes $y-a_1$ or $y-a_2$. The valuation of a
polynomial $\alpha\in\mathbb R[y]$ at the point at infinity is equal to
$\deg{\alpha}$. The above discussion shows that if $\alpha$ and $\beta$ are
distinct monic irreducible polynomials in $\mathbb R[y]$, then 
\begin{equation}
\label{eq12}
\left(\alpha/\beta\right)\left(\beta/\alpha\right)
=\left(-1\right)^{\deg{\alpha}\deg{\beta}}\text{ .}
\end{equation}
\end{example}

\begin{example}
\label{ex8}
Let $f(y)$ be any polynomial in $\mathbb R[y]$ and set $\alpha=y^2-1$,
$\beta=y+(y^2-1)f(y)$. This example is related to a question that came up
in a seminar
presentation at F.A.U. during the fall semester of 1993  by Jim Brewer,
on the subject of  Linear Control  Theory over a commutative ring.
The problem was to 
show that there exists 
a reachable feedback control system over $\mathbb R[y]$
which is not coefficient assignable.
The problem was reduced to showing
that the polynomial
$\alpha$ is not a square modulo the polynomial
$\beta$.
First note that this is an easy consequence
of the Intermediate Value Theorem from Calculus. 
Since $\beta(\pm 1)=\pm 1$, there exists a real number $\xi$ between -1 and
1 such that $\beta(\xi)=0$. Now $\alpha(\xi) < 0$, hence $\alpha$ is not a
square modulo $\beta$.
Now we prove the same result using \eqref{eq11}.
At infinity, $\alpha$ is a square, hence $(\alpha,\beta)_2$ ramifies only at
prime divisors of $\alpha$ or $\beta$. At the primes $y\pm 1$ dividing
$\alpha$ we see that \eqref{eq3} becomes $\pm 1$. So $(\alpha,\beta)_2$
is ramified at the prime $y+1$ and unramified at the prime $y-1$. But
the exact sequence \eqref{eq11} implies that the ramifications ``sum to
zero''. So $(\alpha,\beta)_2$ ramifies at some prime divisor corresponding
to a zero of $\beta$. By Theorem~\ref{th3}, 
$\alpha$ is not a square modulo $\beta$.
\end{example}

\begin{example}
\label{ex0}
This is a generalization of Example~\ref{ex8}. It comes from
\cite[Lemma~1]{BFKS:Wdr}. 
Let $q$ be a prime number and $k$ any field with characteristic different
{f}rom $q$. Let $\omega$ be
a unit in $k$ which is not a $q$th power. Assume $k$ contains a primitive
$q$th root of unity $\zeta$.

We apply Lemma~\ref{lem1} and Theorem~\ref{th1} to the curve $X= \mathbb P^1_k
= \Proj{k[x_0,x_1]}$. Let $K = K(X)$. Dehomogenize with respect to $x_1$,
set $y = x_0/x_1$ and view $K$ as $k(y)$. Set
\[ \alpha = (y-1)^{q-1}(y-\omega) \]
and
\[ \beta = y + (y-1)^{q-1}(y- \omega)f(y) \]
where $f(y)$ is an arbitrary polynomial in $k[y]$.
We will show that $\alpha$ is not a $q$th power modulo $\beta$. The proof
amounts to forcing a $q$th degree reciprocity law out of Theorem~\ref{th1}.

Consider the symbol algebra
$(\alpha,\beta)_q$ as a class in ${_qB(K)}$.
We show that $(\alpha,\beta)_q$ is nontrivial (is not in $\ker{a}$) and has
nontrivial ramification.  Let $x$ be the closed point of
$X$ where $y=\omega$. At the point $x$, the residue field is $k$ and the
ramification of $(\alpha,\beta)_q$ corresponds to the field extension
$k(1/{\omega}^{1/q})$, which represents an element of order $q$ in
$\HH^1(k(x),\mathbb Q/\mathbb Z)$. By Lemma~\ref{lem1}, $\HH^1(k(x),\mathbb Q/\mathbb Z)
\xrightarrow{r}  \HH^3(X,\mathbb G_m)$ is injective. However in \eqref{eq7}, $r
\circ a$ is 
the zero map. So there is another closed point $x' \not = x$ such that the
symbol algebra $(\alpha,\beta)_q$ ramifies at $x'$. Notice that
$(\alpha,\beta)_q$ is unramified at ``the point at infinity'' corresponding
to $x_1 =0$. This is because when $x_1=0$, $\alpha$ is a $q$th power hence the
tame symbol \eqref{eq3} is a $q$th power. At the point corresponding to the
other prime factor $y-1$ of $\alpha$, we see that $\beta$ is equivalent to
1, hence is a $q$th power. So $(\alpha,\beta)_q$ is unramified at $y-1$ also.
By a process of elimination, the symbol algebra $(\alpha,\beta)_q$
necessarily ramifies at a point corresponding to a prime divisor $g(y)$ of
the polynomial $\beta$. 
By Theorem~\ref{th3},
$\alpha$ is not a
$q$th power modulo $\beta$.
\end{example}

\begin{example}
\label{ex4}
Let $k= \mathbb Q$, $X = \mathbb P^1_k = \Proj{k[x_0,x_1]}$ and $K = K(X)$.
Dehomogenize with respect to $x_1$, set $y=x_0/x_1$ and view  $K = \mathbb
Q(y)$.  Choose $p $ in $\mathbb Z$ such that $p $ is not a square in $\mathbb
Q[y]/(y^2+1)$. Consider the symbol algebra $D=(p,y^2+1)_2$ over $K$. Then
$D$ ramifies at the point $x \in X_1$ where $y^2+1=0$. But this is the only
point where $D$ ramifies. So in the sequence of Theorem~\ref{th1},
the map  $\HH^1(k(x),\mathbb Q/\mathbb Z) \to \HH^3(X,\mathbb G_m)$ is not injective.
\end{example}

\begin{example}
\label{ex5}
Let $k$ be an algebraically closed field of characteristic different from
2. Let $F = k(T)$ where $T$ is an indeterminate. Set $X = \mathbb P^1_F =
\Proj{F[x_0,x_1]}$ and $K = K(X)$. Dehomogenize with respect to $x_1$, set
$y=x_0/x_1$ and view $K$ as $k(T)(y)$. Consider the symbol algebra
$(T,y^2-T(T^2-1))_2$ over $K$. This algebra ramifies at the point $x \in
X_1$ where $y-T(T^2-1)=0$ (the proof is identical to the one given in
Example~\ref{ex1} which follows).
At the point at infinity $x_1=0$, and the symbol
is $(T,x_0^2)_2$ which is split. So in the sequence of Theorem~\ref{th1},
we see that when $k(x) \neq F$ the map $\HH^1(k(x),\mathbb Q/\mathbb Z) \to
\HH^3(X,\mathbb G_m)$ is not injective. 

This example is not fair because the left hand side of the symbol is a unit
on $X$, i.e. is in $\HH^0(X,\mathbb G_m)$. But rotating the equation for the
elliptic curve gives an example which is fair. Consider the symbol algebra
$(y-T, (y+T)^2-(y-T)((y-T)^2-1))_2$ over $K$. This algebra ramifies at the
point $x \in X_1$ where $(y+T)^2-(y-T)((y-T)^2-1)=0$. It is unramified at
the point at infinity on $X$ and at the point where $y-T=0$. So we get the
same conclusion as before.
\end{example}

\section{The Two-dimensional Case.}
In this section we consider some cases where $q$th degree reciprocity works 
for the function field of a 2-dimensional scheme.
All of the results in this section are a consequence of the following
theorem  due to M. Artin and D. Mumford. Throughout this section $k$ is an
algebraically closed field of characteristic $p$ and we always
work modulo $p$-groups ($p=0$ is allowed). In this section $X$ will be a
nonsingular integral algebraic surface over $k$.

\begin{theorem}
\label{th2}
If $X$ is a nonsingular integral surface over $k$ and $K=K(X)$ is the
function field of $X$, then the sequence
\[
0 \to \B(X) \to \B(K) \xrightarrow{a} \bigoplus_{C \in X_1}\HH^1(K(C),\mathbb Q/\mathbb Z)
\xrightarrow{r} \bigoplus_{p \in X_2} \mu(-1) \xrightarrow{S} \HH^4(X,\mu) \to 0
\]
is a complex which is exact except that $\im(a) \not = \ker(r)$ in general.
If $\HH^3(X,\mu)=0$ (true for example if $X$ is affine, or complete and
simply connected), the sequence is exact. The map $a$ is the ``ramification
map'' \eqref{eq2}.
\end{theorem}
\begin{proof} Follows from combining sequences (3.1) and (3.2) of
\cite[p.~86]{AM:See}.
\end{proof}

In Theorem~\ref{th2}, the fact that $r \circ a$ is the zero map can be
thought of as a quadratic reciprocity law for elements of order 2 or a
$q$th degree reciprocity law for elements of order $q$.  However, as the
following example shows, the map
$r$ sometimes has a nontrivial kernel at a prime divisor $C$.

\begin{example}
\label{ex1}
If $C$ is an irreducible curve on $X$, then the group $\HH^1(C,\mathbb Z/2)$
becomes an obstruction to a quadratic reciprocity law for $K(X)$. For
example, let $X = \mathbb A^2_k = \Spec{k[x,y]}$ and $K = k(x,y)$ where $k$
has characteristic different from 2. Let $f=x$,
$g=y^2-x(x^2-1)$. The Legendre symbols have values $(g/f) = 1$ and $(f/g)
= -1$. Here is a proof that $(f/g)= -1$. The curves $f=0$ and $g=0$
intersect at 2 points: $P$, the point where $x=y=0$ and $Q$, the point at
infinity. Think of $f$ as a function on the elliptic curve $C$ defined by
the equation $g=0$.  The divisor of $f$ on $C$ is $2P-2Q$. If $f$ is a
square on $C$, then $f=h^2$ for some function $h$ on $C$. The divisor of
$h$ is $P-Q$ which is not the divisor of a function because $C$ is not a
rational curve \cite[p. 138]{H:AG}. So the algebra $(f,g)_2$ is nontrivial
over $K$ and ramifies exactly along the elliptic curve $C$ with equation
$g=0$.  The ramification
data along $C$ for $(f,g)_2$ is the extension $K(C)(\sqrt{f})$ which is an
unramified quadratic extension of $K(C)$, hence represents a class in
$\HH^1(C,\mathbb Q/\mathbb Z)$. The ramification divisor of $(f,g)_2$ is the prime
divisor $C$.
\end{example}

The obstruction to a $q$th degree reciprocity law illustrated by
Example~\ref{ex1} is overcome by localizing in the \'etale topology at a
closed point. Let $p \in X_2$ be a closed point on $X$ and let
$\ScrO^h_{X,p}$
denote the henselization of $\ScrO_{X,p}$. Let $K^h$ denote
the quotient field of $\ScrO^h_{X,p}$ and $X^h = \Spec{\ScrO^h_{X,p}}$.
{F}rom the proof of Theorem~\ref{th2}, it follows that the sequence
\begin{equation}
\label{eq9}
0 \to  \B(K^h) \xrightarrow{a} \bigoplus_{C \in X_1}\HH^1(K(C),\mathbb Q/\mathbb Z)
\xrightarrow{r}  \mu(-1) \to 0
\end{equation}
is exact.
The reader is referred to \cite[Theorem~(1.2)]{A:Tdo} for a proof of a
version of \eqref{eq9} in which $X$ is only assumed to be normal with
rational singularities.
In this case, each curve $C \in X_1^h$ is a henselian curve with
1 closed point and $\HH^1(K(C),\mathbb Q/\mathbb Z) \cong \mu(-1)$. The map $r$ is
an isomorphism on each summand \cite[p. 86]{AM:See}. 
The sequence \eqref{eq9} also holds if instead of
henselizing $\ScrO_{X,p}$ we complete with respect to the maximal ideal. 
In particular,  there is the following weak version of a reciprocity formula for
power series in 2 variables over $k$.

\begin{proposition}
\label{prop1}
Let $k$ be an algebraically closed field and $q$ a prime number different
{f}rom the characteristic of $k$. Let $f$ and $g$ be nonzero irreducible
power series in $k[[x,y]]$. If  $f$ is a $q$th power modulo $g$, then the
residue class of $g$ is a $q$th power in the normalization of
$k[[x,y]]/(f)$. There exist functions $s$, $t$ in $k((x,y))$ satisfying
$g-s^q=ft$. 
\end{proposition}
\begin{proof}
The ramification divisor of the symbol algebra $(f,g)_q$ is a subset of the
divisor of $fg$. By Theorem~\ref{th3}, $(f,g)_q$ ramifies at most along the
divisor $C$ of $f$. By \eqref{eq9}, $(f,g)_q$ is unramified at each prime
divisor. So the tame symbol is a $q$th power. That is, $g$ represents a
$q$th power in the field of fractions $K(C)$ of $\ScrO(C)=k[[x,y]]/(f)$. So
there are elements $s$, $t$ in $k((x,y))$ satisfying $s^q-g=ft$.
The function $s$ represents a class in $K(C)$ that is integral over
$\ScrO(C)$. 
\end{proof}

The next example shows that quadratic reciprocity for power series is
hindered by the fact that the ring $k[[x,y]]/(f)$ is not necessarily
factorial. This problem occurs when the curve defined by $f=0$ is singular.

\begin{example}
\label{ex9}
Let $f=x$, $g=y^2-x^3$ be power series in $k[[x,y]]$ and assume the
characteristic of $k$ is
different from 2. 
Since $\left( g/f \right)=1$, 
by Proposition~\ref{prop1}, $f$ is a square in $K(C)$,
where $C$ is 
the cubic curve  with equation $g=0$. 
That is, there are
functions $s$, $t$ in $k((x,y))$ satisfying $x-s^2=(y^2-x^3)t$.
In fact, one can check that $s=y/x$ and $t= -1/x^2$ work. This equation
also shows that $s$ is integral over $\ScrO(C)$ hence is in the
normalization $\ScrO(\bar{C})$. Now the curve $C$ has a cusp singularity
and $\ScrO(C)$ is non-normal.
Since adjoining $y/x$ to $\ScrO(C)$ results
in a normal ring, we see that $f$ is not a square in $\ScrO(C)$.
\end{example}

\begin{corollary}
\label{cor1}
If, in the context of Proposition~\ref{prop1}, the lowest degree form of
$f$ has degree $\le 1$ and $f$ is a $q$th power modulo $g$,
then $g$ is a $q$th power modulo $f$.
\end{corollary}
\begin{proof}
If the lowest degree form of $f$ has degree 0, then $f$ is invertible and
the corollary is true. If the lowest degree form of $f$ is linear, then the
divisor of $f$ is nonsingular so $\ScrO = k[[x,y]]/(f)$ is a discrete
valuation ring. Let $K$ denote the field of fractions of $\ScrO$. If $g$ is
a $q$th power in $K$, then $g$ is a $q$th power in $\ScrO$.
\end{proof}

\begin{example}
\label{ex2}
Consider the polynomials $f=x$, $g=y^2-x(x^2-1)$ from Example~\ref{ex1},
but this time view them as elements of the power series ring $k[[x,y]]$
over $k$.
Since $(g/f)=1$, from Corollary~\ref{cor1} we have $(f/g)=1$.
So $f$ is a square modulo $g$ in $k[[x,y]]$. In other words, there are
power series $s(x,y)$, $t(x,y)$ in $k[[x,y]]$ satisfying the equation
$x -  s ^2 = \left( y^2 - x(x^2-1) \right) t $.
Notice that this is contrary to the value of $(f/g)$ in the polynomial
ring. The reason of course is that the unramified cyclic extensions of the
elliptic curve $k[x,y]/(g)$ have been split by completion. This includes
the extension corresponding to adjoining $\sqrt{x}$.
\end{example}

In order to alleviate the obstruction to $q$th degree reciprocity it is not 
necessary to localize at a closed point in $X_2$. It is sufficient to
localize along a curve $C \in X_1$ such that $\HH^1(C,\mathbb Q/\mathbb Z) = 0$.
For simplicity assume $X = \Spec{R}$ where $R$ is the coordinate ring of an
affine nonsingular integral 2-dimensional variety over $k$. Let $I$ be an
ideal in $R$ such that $R/I$ is reduced and connected. Let
$(\tilde{R},\tilde{I})$ denote the henselization of $R$ along $I$. For the
basic properties of henselian couples, the reader is referred to
\cite{MR0277519}. Denote by $\hat{R}$ the completion of $R$ with respect to the
ideal $I$. Let $\tilde{K}$ be the quotient field of $\tilde{R}$ and
$\hat{K}$ the quotient field of $\hat{R}$. From the proof of
Theorem~\ref{th2} (see \cite{F:Daon}), the sequence
\begin{equation}
\label{eq10}
0 \to \B(\tilde{K}) \xrightarrow{a} \bigoplus_{C \in \tilde{X}_1}\HH^1(K(C),\mathbb Q/\mathbb Z)
\xrightarrow{r} \bigoplus_{p \in \tilde{X}_2} \mu(-1)  \to 0
\end{equation}
is exact where $\tilde{X} = \Spec{\tilde{R}}$.
Sequence \eqref{eq10} is also exact for $\hat{K}$, $\hat{X}$
replacing $\tilde{K}$, $\tilde{X}$. If the curve $R/I$ has the property
that each irreducible component $C$ is simply connected, then $\HH^1(C,\mathbb
Q/\mathbb Z)=0$.
As a special case, consider the following.

\begin{proposition}
\label{prop2}
Let $k$ be an algebraically closed field and $q$ a prime number different
{f}rom the characteristic of $k$. 
Let $f$ and $g$ be nonzero irreducible power series in $y$ with
coefficients that are polynomials in $x$. 
If  $f$ is a $q$th power modulo $g$, then the
residue class of $g$ is a $q$th power in the normalization of
$k[x][[y]]/(f)$. There exist functions $s$, $t$ in the field of fractions of
$k[x][[y]]$ satisfying $g-s^q=ft$. 
\end{proposition}
\begin{proof} The ring $k[x][[y]]$ is the completion of $R=k[x,y]$ with
respect to the ideal $I=(y)$. The curve $R/I=k[x]$ is simply connected,
hence $\HH^1(k[x],\mathbb Q/\mathbb Z)=0$. For each curve $C$ in $\hat{X}_1$, $C$
is simply connected hence $K(C)$ has only ramified cyclic extensions. 
The rest is as for Proposition~\ref{prop1}.
\end{proof}

\begin{corollary}
\label{cor2}
Suppose, in the context of Proposition~\ref{prop2}, 
the curve defined by $f=0$ is nonsingular.
If $f$ is a $q$th power modulo $g$, then
$g$ is a $q$th power modulo $f$.
\end{corollary}
\begin{proof} There are 2 possibilities for the curve $C$ defined by $f=0$. If
$C$ is the curve $y=0$, then the ring $\ScrO=k[x][[y]]/(f)$ is isomorphic
to $k[x]$. Otherwise $C$ is a henselian curve and
$\ScrO$ is a local principal ideal domain, since $C$ is nonsingular.
In both cases, $\ScrO$ is normal, so $g$ is a $q$th power in $\ScrO$.
\end{proof}

\begin{example}
\label{ex3}
Once again
consider the polynomials $f=x$, $g=y^2-x(x^2-1)$ from Examples~\ref{ex1}
and \ref{ex2},
but this time view them as elements of the power series ring in $y$ with
coefficients in $k[x]$, $k[x][[y]]$.
Notice
that $g$ factors into $x(x-1)(x+1)$ in $k[x,y]/(y)$, so in $k[x][[y]]$ $g$
factors into a product of 3 irreducibles.
(This is the henselian property.)
Denote this factorization by
$g=g_1g_2g_3$ where $g_1$ corresponds to $x=0$, $g_2$ to $x=1$ and $g_3$ to
$x=-1$.
Each curve $g_i=0$ is nonsingular.
Upon completion
with respect to $(y)$, the elliptic curve
$k[x,y]/(g)$ splits into a direct sum of 3 complete discrete valuation rings
corresponding to the 3 points $x=0$, $x=1$ and $x=-1$.
Now $g$ is clearly a square modulo $f$. Since $g_2$ and $g_3$ are units
modulo $f$ and $k[[y]]$ is a complete local ring and $k$ is algebraically
closed, $g_2$ and $g_3$
are also squares modulo $f$. This implies that $g_1$ is a square modulo
$f$. So far we have $(g_1/f)=(g_2/f)=(g_3/f)=1$.
{F}rom Corollary~\ref{cor2} we have $(f/g_1)=(f/g_2)=(f/g_3)=1$.
It follows that $f$ is a square in $k[x][[y]]/(g)$.
In other words, there are
power series $s$, $t$ in $k[x][[y]]$ satisfying the equation
\(
x -  s^2 = \left( y^2 - x(x^2-1) \right) t 
\).
\end{example}

\begin{example}
\label{ex6}
Once again
consider the polynomials $f=x$, $g=y^2-x(x^2-1)$ from Examples~\ref{ex1},
\ref{ex2}, and \ref{ex3},
but this time view them as elements of the power series ring in $x$ with
coefficients in $k[y]$, $k[y][[x]]$.
Notice
that in the ring $k[y][[x]]$, $g$
is irreducible.
Since the curve $g=0$ is nonsingular,
Corollary~\ref{cor2} applies.
By Corollary~\ref{cor2}, $(f/g)=(g/f)=1$.
In other words  the equation
\(
x - s^2 = \left( y^2 - x(x^2-1) \right) t
\)
has a solution for
power series $s$, $t$ in $k[y][[x]]$.
\end{example}

\begin{example}\label{ex10}
Let $f=x$, $g=y^2-x^2(x-1)$ viewed as power series in $k[[x,y]]$. Then $g$
factors into 2 irreducibles, say $g=g_1g_2$. The curves $g_1=0$ and $g_2=0$
correspond to the 2 branches through the origin on the nodal cubic curve
$g=0$. So $-1=(g_i/f)$ and by Corollary~\ref{cor1} $(f/g_i)=-1$. This
implies $(f/g)=-1$. If we view $f$, $g$ as elements of the subring
$k[x][[y]]$ or $k[y][[x]]$, we have similar results, namely $(g/f)=1$ and
$(f/g)=-1$ in each case.
\end{example}


\begin{thebibliography}{10}

\bibitem{AM:See}
M.~Artin and D.~Mumford, \emph{Some elementary examples of unirational
  varieties which are not rational}, Proc. London Math. Soc. (3) \textbf{25}
  (1972), 75--95. \MR{0321934 (48 \#299)}

\bibitem{A:Tdo}
Michael Artin, \emph{Two-dimensional orders of finite representation type},
  Manuscripta Math. \textbf{58} (1987), no.~4, 445--471. \MR{88m:16032}

\bibitem{AB:Bgd}
Maurice Auslander and Armand Brumer, \emph{Brauer groups of discrete valuation
  rings}, Nederl. Akad. Wetensch. Proc. Ser. A \textbf{30} (1968), 286--296.
  \MR{37 \#4051}

\bibitem{BFKS:Wdr}
J.~Brewer, T.~Ford, L.~Klingler, and W.~Schmale, \emph{When does the ring
  ${K}[y]$ have the coefficient assignment property?}, J. Pure Appl. Algebra
  \textbf{112} (1996), no.~3, 239--246. \MR{97g:13014}

\bibitem{D:A}
Max Deuring, \emph{Algebren}, Springer-Verlag, Berlin, 1968. \MR{37 \#4106}

\bibitem{FS:Bgrff}
B.~Fein and M.~Schacher, \emph{Brauer groups of rational function fields over
  global fields}, The {B}rauer group (Sem., Les Plans-sur-Bex, 1980), Springer,
  Berlin, 1981, pp.~46--74. \MR{82h:12025}

\bibitem{F:BgLp}
T.~J. Ford, \emph{On the {B}rauer group of a {L}aurent polynomial ring}, J.
  Pure Appl. Algebra \textbf{51} (1988), no.~1-2, 111--117. \MR{89h:13006}

\bibitem{F:Daon}
Timothy~J. Ford, \emph{Division algebras over nonlocal {H}enselian surfaces},
  Pacific J. Math. \textbf{147} (1991), no.~2, 301--310. \MR{92d:12011}

\bibitem{MR611868}
Ofer Gabber, \emph{Some theorems on {A}zumaya algebras}, The {B}rauer group
  ({S}em., {L}es {P}lans-sur-{B}ex, 1980), Lecture Notes in Math., vol. 844,
  Springer, Berlin, 1981, pp.~129--209. \MR{611868 (83d:13004)}

\bibitem{MR0197380}
Carl~Friedrich Gauss, \emph{Disquisitiones arithmeticae}, Yale University
  Press, New Haven, Conn.-London, 1966, Translated into English by Arthur A.
  Clarke, S. J. \MR{0197380}

\bibitem{G:GBIII}
Alexander Grothendieck, \emph{Le groupe de {B}rauer. {I}{I}{I}. {E}xemples et
  compl\'ements}, Dix Expos\'es sur la Cohomologie des Sch\'emas,
  North-Holland, Amsterdam, 1968, pp.~88--188. \MR{0244271 (39 \#5586c)}

\bibitem{H:AG}
Robin Hartshorne, \emph{Algebraic geometry}, Springer-Verlag, New York, 1977.
  \MR{0463157 (57 \#3116)}

\bibitem{MR770588}
Raymond~T. Hoobler, \emph{Functors of graded rings}, Methods in ring theory
  ({A}ntwerp, 1983), NATO Adv. Sci. Inst. Ser. C Math. Phys. Sci., vol. 129,
  Reidel, Dordrecht, 1984, pp.~161--170. \MR{770588 (86f:16008)}

\bibitem{M:EC}
James~S. Milne, \emph{\'{E}tale cohomology}, Princeton University Press,
  Princeton, N.J., 1980. \MR{559531 (81j:14002)}

\bibitem{MR0277519}
Michel Raynaud, \emph{Anneaux locaux hens\'eliens}, Lecture Notes in
  Mathematics, Vol. 169, Springer-Verlag, Berlin-New York, 1970. \MR{0277519
  (43 \#3252)}

\bibitem{R:MO}
I.~Reiner, \emph{Maximal orders}, Academic Press [A subsidiary of Harcourt
  Brace Jovanovich, Publishers], London-New York, 1975, London Mathematical
  Society Monographs, No. 5. \MR{0393100 (52 \#13910)}

\bibitem{S:TV}
O.~F.~G. Schilling, \emph{The {T}heory of {V}aluations}, American Mathematical
  Society, New York, N. Y., 1950. \MR{13,315b}

\bibitem{MR554237}
Jean-Pierre Serre, \emph{Local fields}, Graduate Texts in Mathematics, vol.~67,
  Springer-Verlag, New York, 1979, Translated from the French by Marvin Jay
  Greenberg. \MR{554237 (82e:12016)}

\end{thebibliography}

\providecommand{\bysame}{\leavevmode\hbox to3em{\hrulefill}\thinspace}
\providecommand{\MR}{\relax\ifhmode\unskip\space\fi MR }
\providecommand{\MRhref}[2]{%
  \href{http://www.ams.org/mathscinet-getitem?mr=#1}{#2}
}
\providecommand{\href}[2]{#2}

\end{document}